\documentclass[hidelinks=true,12pt]{article}
\pdfoutput=1

\usepackage{pdfpages}

\usepackage{amsmath,amsthm,amsfonts,amssymb,times,latexsym,enumerate,comment,marginnote,slashed,dsfont}
\usepackage{pdfpages, relsize, setspace}
\usepackage{xypic}
\usepackage[OT2,T1]{fontenc} 
\usepackage{url}

\DeclareSymbolFont{cyrletters}{OT2}{wncyr}{m}{n}
\DeclareMathSymbol{\sha}{\mathalpha}{cyrletters}{"58}

\newcommand{\midskip}{\hspace{0.1in}}
\newcommand{\abs}[1]{\left\lvert#1\right\rvert}
\newcommand{\norm}[1]{\lVert#1\rVert}

\newcommand{\mbf}{\mathbb{F}}
\newcommand{\mbr}{\mathbb{R}}
\newcommand{\mbz}{\mathbb{Z}}

\newcommand{\mbc}{\mathbb{C}}
\newcommand{\mbq}{\mathbb{Q}}
\newcommand{\mba}{\mathbb{A}}

\newcommand{\mbp}{\mathbb{P}}
\newcommand{\mbo}{\mathcal{O}}

\renewcommand{\:}{\colon}

\newcommand{\ra}{\rightarrow}

\newcommand{\iso}{\cong}

\newcommand{\bs}{\backslash}
\newcommand{\ssq}{\subseteq}

\DeclareMathOperator{\Pic}{Pic}
\DeclareMathOperator{\Div}{Div}
\DeclareMathOperator{\Princ}{Princ}

\newcommand{\wtilde}{\widetilde}

\DeclareMathOperator{\ord}{ord}
\DeclareMathOperator{\Aut}{Aut}

\DeclareMathOperator{\Gal}{Gal}

\renewcommand{\div}{\operatorname{div}}

\newcommand*{\longhookrightarrow}{\ensuremath{\lhook\joinrel\relbar\joinrel\rightarrow}}
\newcommand*{\inj}{\longhookrightarrow}

\newcommand{\bbm}{\begin{bmatrix}}
\newcommand{\ebm}{\end{bmatrix}}
\newcommand{\bpm}{\begin{pmatrix}}
\newcommand{\epm}{\end{pmatrix}}
\newcommand{\xym}[1]{\xymatrix{#1}}

\newcommand{\gen}[1]{\langle #1 \rangle}

\newcommand{\allstar}[1]{\[\begin{aligned} #1  \end{aligned}\]}

\newcommand{\set}[1]{\left \{#1 \right \}}

\newtheorem{theorem}{Theorem}[section]
\newtheorem{lemma}[theorem]{Lemma}
\newtheorem{proposition}[theorem]{Proposition}
\newtheorem*{theorem*}{Theorem}
\newtheorem{corollary}[theorem]{Corollary}

\newtheorem{question}[theorem]{Question}

\theoremstyle{definition}
\newtheorem{definition}[theorem]{Definition}

\newtheorem{remark}[theorem]{Remark}
\newtheorem{example}[theorem]{Example}

\usepackage{listings}

\newcommand{\newtext}[1]{{#1}}

\DeclareMathOperator{\rank}{rk}
\DeclareMathOperator{\ranktwo}{rk_2}
\DeclareMathOperator{\Spec}{Spec}
\DeclareMathOperator{\Cl}{Cl}
\DeclareMathOperator{\id}{id}

\newcommand{\E}{\mathcal{E}}
\newcommand{\Bl}{\mathrm{Bl}}
\newcommand{\Sp}{\operatorname{Sp}}

\newcommand{\aff}{\mathrm{aff}}
\newcommand{\al}{\mathrm{al}}
\newcommand{\act}[2]{{}^{#1}\!#2}

\newcommand{\supellname}{\newtext{$\mu_3$-branched cover of $\mbp_\mbq^1$} }
\newcommand{\riemannrochspace}[1]{L(#1)}

\let\svthefootnote\thefootnote

\title{An explicit family of cubic number fields with large $2$-rank of the class group}
\author{Kulkarni, Avinash	
	}
\date{\today}

\begin{document}

\maketitle

\begin{abstract}
	We show how to construct infinite families of explicitly determined cubic number fields whose class group has a subgroup isomorphic to $(\mbz/2)^8$ using degree $1$ del Pezzo surfaces. We illustrate the method and provide an example of such a family.%
	\let\thefootnote\relax
	\footnote{The author was supported by NSERC.}	
	\addtocounter{footnote}{-3}\let\thefootnote\svthefootnote
\end{abstract} 

\section{Introduction}

The task of finding number fields with exotic class groups is an old problem, receiving attention as far back as 1922. In that year, Nagell proved that there are infinitely many quadratic number fields whose class groups have a subgroup isomorphic to $\newtext{(\mbz/n)}$ \cite{nagel1922klassenzahl}. \newtext{The special case of class groups of quadratic number fields has been an active research topic since then \cite{ankeny1955divisibility, honda1968real, yamamoto1970unramified, weinberger1973quadratic}. Class groups for non-quadratic number fields have also been considered \cite{nakano1986construction}, including a construction for cubic number fields whose class groups have $2$-rank at least $6$ \cite{nakano1986construction}. A modern trend is to establish quantitative estimates for the number of number fields of a fixed degree and bounded discriminant whose class group has a special subgroup \cite{murty1999exponents, bilu2005divisibility}.} Number fields of degree indivisible by $p$ whose class group has a subgroup isomorphic to $(\mbz/p)^r$ for $r>2$ are difficult to exhibit in practice; even though it is conjectured by Cohen and Lenstra that these number fields have positive density (when ordered by discriminant with respect to the usual density) \cite{cohen-lenstra}. 

A relatively recent idea to construct infinite families of such number fields (of degree $d$) is to construct an algebraic $d$-gonal curve whose Picard group is rigged to have either $\mu_p^r$ or $(\mbz/p)^r$ as a subgroup (depending on the method) and then to choose fibres over points satisfying certain local conditions. See for example \cite{gile12}, \cite{mes83}. The examples with the best yield of class group $p$-rank from Picard group $p$-rank have been curves fibred over $\mbp_\mbq^1$ with a totally ramified fibre. 

We explicitly construct an infinite family of cubic number fields and prove in Corollary \ref{cor: main result} that the class group of \newtext{``most'' members of } this family has $2$-rank at least $8$. To do this, we use a connection between degree 1 del Pezzo surfaces and trigonal genus 4 curves outlined in Section 2 which allows us good control over both the $2$-torsion in the Jacobian and the ramification of the trigonal map. Finally, in Section 4 we make some remarks on the $2$-rank bounds produced by this method for families of Galois cubic number fields.

\bigskip
\noindent
\textbf{Acknowledgements.} I would like to thank Nils Bruin for his guidance in writing this article. \newtext{I would also like to thank the anonymous referee for numerous helpful suggestions.}

\section{Del Pezzo surfaces, elliptic surfaces, and curves of genus 4}

For a field $k$ of characteristic 0, we denote an algebraic closure of $k$ by $k^{\al}$. For a finitely generated abelian group $M$, we denote the rank as a $\mbz$-module by $\rank M$. Additionally, we denote the $2$-torsion of $M$ by $M[2]$ and denote $\ranktwo M[2]$ to be the rank of $M[2]$ as a $(\mbz/2)$-module. Following \cite{zar08}, we say a collection of 8 distinct points of $\mbp^2$ are in \emph{general position} if no three are co-linear, no six lie on a conic, and any cubic passing through all of these points is non-singular at each of these points. If $X$ is a \newtext{smooth} variety defined over a field $k$ then we denote the class of the canonical divisor by $\kappa_X$ and call $-\kappa_X$ the \emph{anti-canonical} class. By abuse of notation we will often denote a canonical divisor of $X$ by $\kappa_X$. If $X$ is a variety defined over $k$ and $Z$ is a positive codimension subvariety of $X$ defined over $k$ then we denote the \emph{blow-up} of $X$ at $Z$ by $\Bl_Z X$. We call the canonically determined birational map from $X$ to $\Bl_Z X$ the \emph{blow-up (at $Z$) map} and the canonically determined morphism from $\Bl_Z X$ to $X$ the \emph{blow-down}. 

We state some facts regarding the connection between del Pezzo surfaces and certain genus 4 curves. A \emph{del Pezzo} surface of degree 1 is a complete smooth surface which is isomorphic over an algebraically closed field to the blow up of $\mbp^2$ at $8$ points in general position. 

\newtext{For a degree 1 del Pezzo surface $X$, the Riemann-Roch space $\riemannrochspace{-3\kappa_X}$} gives a natural description of $X$ as a sextic hypersurface in the weighted projective space $\mbp(1:1:2:3)$. \newtext{If the characteristic of $k$ is not equal to $2$ then} we can write
	\[
	X \: z^2 = c_0 w^3 + c_2(x,y) w^2 + c_4(x,y) w + c_6(x,y)
	\]
with each $c_m$ homogeneous in $x,y$ of degree $m$ and $c_0 \neq 0$. Every degree 1 del Pezzo surface admits an involution called the \emph{Bertini involution}. On the model above this corresponds to the natural map $(x:y:w:z) \mapsto (x:y:w:-z)$. \newtext{The linear system $\abs{-2\kappa_X}$ } has a single base point and on the model above this is the point $(0:0:c_0:c_0^2)$. \newtext{From the model above, we see that the fixed locus of the Bertini involution is the union of the point $(0:0:c_0:c_0^2)$ and the subvariety of $X$ where $z$ vanishes. The vanishing locus of $z$ has dimension $1$.}

We rely on a classical result connecting the arithmetic of a del Pezzo to the curve fixed by the Bertini involution. We refer the reader to \cite[Section 2.2, Lemma 2.4, Lemma 2.7, Lemma 4.2]{zar08} or \cite[Theorem 4.4]{vak01} for the proof.

\begin{proposition}[\cite{zar08}, \cite{vak01}] \label{del-pezzo}

Let $X$ be a del Pezzo surface of degree 1 defined over a characteristic 0 field $k$ and let $\iota \: X \ra X$ be the Bertini involution and let $C$ be the \newtext{irreducible $1$-dimensional subvariety of the} fixed locus of $\iota$. Let $O$ be the base point of \newtext{the linear system $\abs{-2\kappa_X}$}.

\begin{enumerate}[(a)]

\item
The blow-up $\E$ of $X$ at $O$ is an elliptic surface over $\mbp^1$ with identity section corresponding to the exceptional curve over $O$. Write $f\: \E \ra \mbp^1$.

\item
We have that $C$ is a smooth non-hyperelliptic irreducible curve of genus 4. The strict transform of $C$ in $\E$ is the multi-section $\E[2] \bs \id_\E$. Moreover, the restriction of $f$ to $C$ is a morphism of degree 3.

\item
Denote by $\bar X$ the base-change of $X$ to $k^{\al}$. Each exceptional curve on $\bar X$ corresponds to a unique class in $\Pic \bar X$. Moreover, $\Pic \bar X$ is generated by the exceptional curves on $X$. Furthermore, $\Pic \bar X$ is generated by a set of 8 pairwise orthogonal exceptional curves and the canonical class.

\item
Every exceptional curve on $\bar X$ restricts to an odd theta characteristic of $C$ and the \newtext{anti-}canonical divisor of $X$ restricts to an even theta characteristic of $C$.

\item
$\rank \Pic(X)(\mbq) - 1 \leq \ranktwo J_C[2](\mbq)$

\end{enumerate}

\end{proposition}

\begin{corollary} \label{reduction to ramification}
Let $X$ be a degree 1 del Pezzo surface defined over a number field $k$ and let $f\: C \ra \mbp_k^1$ be the degree 3 morphism from Proposition \ref{del-pezzo}(b). Then the ramification of $f$ over $p \in \mbp_{k}^1$ is classified by the reduction type of the fibre $\E_p$. In particular, if one of the special fibres of $\E$ has additive reduction then $f$ has a totally ramified fibre.
\end{corollary}

\begin{proof}
 Since $f$ is defined over a number field and $\deg f = 3$ we can determine the set of ramification indices of points in the fibre over $p$ by counting the number of $k^{\al}$-points on a smooth model of $C$. We can also determine the reduction type of $\E_p$ by counting the number of $k^{\al}$-points in $(\E[2] \bs \id_\E)_p$. Since $C$ is isomorphic to $\E[2] \bs \id_\E$ as a subvariety of $\E$ over $\mbp_k^1$ and $C$ is smooth, these two quantities are equal for every $p \in \mbp^1(k^{\al})$.
\end{proof}

The advantage of using curves arising from del Pezzo surfaces is the fact that $\rank \Pic(X)(\mbq)$, and hence $\ranktwo J_C[2](\mbq)$, is easy to control. To be specific, if $X$ is the blow up of eight (rational) points on $\mbp^2$ in general position then $\Pic(X)$ is generated by the (rational) exceptional curves of the blowup and the canonical divisor. In light of Proposition \ref{del-pezzo}, this allows us to build trigonal genus 4 curves with large rational 2-torsion rank. Among these curves, we can easily identify those whose trigonal morphism has a totally ramified fibre.

\begin{corollary} \label{cor: fibres and ramification of branch curve}
	Let $X$ be a del Pezzo surface such that $\rank \Pic(X)(\mbq) = 9$ and let $C$ be the branch curve of the Bertini involution with trigonal morphism $f\: C \ra \mbp^1$. Write $X := \Bl_Z \mbp^2$ where $Z$ is a collection of $8$ rational points in general position. Then there is a one-to-one correspondence between totally ramified points of $f$ and cuspidal cubic plane curves through $Z$. 
\end{corollary}

\begin{proof}
	Let $Y$ be a cuspidal cubic passing through $Z$, let $u$ be a cubic form defining $Y$, and let $P \in Y$ be the cusp. Note the criterion that $Z$ be in general position forces $P \not \in Z$ (points of $Z$ being in general position, there is no cubic passing through the points of $Z$ with a \newtext{singularity} at one of them). 
	
	The eight points of $Z$ determine a pencil of cubic forms $\set{\lambda u + \mu v : (\mu : \lambda) \in \mbp^1}$. Such a pencil has a base locus which consists of 9 points, these being the points of $Z$ as well as an additional point $O$. Intersection multiplicities ensure that $O$ is not a cusp or node of any of the curves in the pencil.

	It follows that the rational map $g(x,y,z) \mapsto (u(x,y,z):v(x,y,z))$ is defined outside of $Z \cup O$, so by the universal property of blowing up there is a morphism $f$ such that the diagram
	\[ 
	\xym{
		& \Bl_{Z \cup O} \mbp^2 \ar[d]^f \ar[dl]_{\phi} \\
		\mbp^2 \ar@{-->}[r]^g & \mbp^1
	}
	\]
	commutes, where $\phi$ is the morphism which is the inverse (as a birational map) to the blow-up map. We have that $\E := \Bl_{Z \cup O} \mbp^2$ is an elliptic surface over $\mbp^1$ whose identity section is the exceptional curve lying over $O$. Because $Y$ is precisely the closure of the locus $\set{Q \in \mbp^2 : g(Q) = (0:1)}$, the strict transform of $Y$ is in fact a fibre of $f$. Since $P \not \in Z \cup O$ we have that the strict transform of $Y$ remains singular (with a cusp) in $\E$. Thus the fibre over $(0:1)$ is a special fibre of $f \: \E \ra \mbp^1$. We see that $\E$ has additive reduction at $(0:1)$, so by Corollary \ref{reduction to ramification} corresponds to a totally ramified point of $f \: C \ra \mbp^1$.
	
	Conversely, if $Q$ is a totally ramified point of $f \: C \ra \mbp^1$ then Corollary \ref{reduction to ramification} shows it is the cusp of a special fibre $\E_p$ whose reduction type is additive. In particular $Q$ does not lie on $\phi^{-1}(Z \cup O)$ so the blow-down of $\E_p$ is a cuspidal plane curve. That it is cubic follows from the fact that every fibre of $\E$ blows down to a cubic curve. 
\end{proof}

We present here for the convenience of the reader the procedure given in \cite{zar08} to get explicit equations for the branch curve. Note that in our example $\mbp^2$ and the del Pezzo surface $X$ are birational over $\mbq$, so we may identify their function fields. We have that $\kappa_X = E_1 + \ldots + E_8 - 3H$ is a canonical divisor for $X$, where $H$ is the pullback of the hyperplane class on $\mbp^2$ and the $E_i$ are the 8 pairwise orthogonal exceptional curves lying over the blown up points of $\mbp^2$. Thus,
	\[
		\newtext{\riemannrochspace{-\kappa_X}} = \gen{u,v}
	\]
where $u,v \in \mbq[x,y,z]$ are cubic forms passing through the 8 base-points of the blow-up. \newtext{Similarly, we have
	\[
		\newtext{\riemannrochspace{-2\kappa_X}} = \gen{u^2,uv,v^2,w}		
	\]
with $w$ a function on $X$ not in the $k$-span of $\set{u^2,uv,v^2}$}. \newtext{We have that $\riemannrochspace{-2\kappa_X}$ } defines a 2-to-1 rational map $\varphi\: \mbp^2 \ra \mbp(1:1:2)$ via
	\[
		\varphi \: (x:y:z) \mapsto (u(x,y,z):v(x,y,z):w(x,y,z))
	\]	
with a unique base-point that is $O$. In fact, if 
	\[
	X \: z^2 = c_0 w^3 + c_2(u,v) w^2 + c_4(u,v) w + c_6(u,v)
	\]
is the model of $X$ in $\mbp(1:1:2:3)$ provided by \newtext{$\riemannrochspace{-3\kappa_X}$} and $\pi\: X \ra \mbp(1:1:2)$ is the natural projection, then $\pi$ and $\varphi$ agree on $X \bs \set{O}$. Since $\pi$, and therefore $\varphi$, is branched along $C$ we can recover a model of the branch curve from a Jacobi criterion. That is,
	\[
	C = V(\newtext{F}) \ssq \mbp^2 \midskip \text{where } \newtext{F} := \det
	\bbm
	u_x & v_x & w_x \\
	u_y & v_y & w_y \\
	u_z & v_z & w_z
	\ebm .
	\] 
In general we have that $C$ is a degree 9 plane curve with order 3 singularities at each of the eight base-points of the blow up (See \cite[Section 5]{zar08}). 

\begin{remark}
Another way to see the eight base-points of the blowup correspond to singular points of the model is as follows. Let $P_1, \ldots, P_8$ be the base-points of the blow up and let $E_1, \ldots, E_8$ be the corresponding exceptional curves lying over these points. For clarity, we denote by $C$ the model of the branch curve on $\mbp^2$ and denote by $C'$ the model of the branch curve on $X$. The strict transform of the blow up of $C$ at the 8 base-points is $C'$. Each exceptional curve $E_i$ of $X$ corresponds to a class of $\Pic X$ that restricts to an odd theta characteristic on $C'$ (by Proposition \ref{del-pezzo}). As a divisor, $E_i$ is effective, so it will restrict to an effective (degree 3) divisor of $C'$. However, each $E_i$ intersects $C'$ in three points (counting multiplicity) which lie over a single point $P_i$ of $C$. We see that $C$ has a singularity of order 3 at $P_i$. 

In other words, once we have computed a model of the branch curve on $\mbp^2$ we can immediately identify effective representatives of 8 odd theta characteristics of $C'$. Additionally, the anti-canonical divisor of $X$ provides us with an even theta characteristic of $C$ by Proposition \ref{del-pezzo}. 
\end{remark}

\begin{example} \label{ex:main example}
	Let $X$ be the del Pezzo surface defined by blowing up $\mbp^2$ at the points
	\begin{gather*}
		(0:-2:1),(3:-9:1),(3:7:1),(8:26:1), \\
		(15:63:1),(24:124:1),(48:342:1),(0:0:1).
	\end{gather*}
	We choose $u,v \in \mbq[x,y,z]$ to be two independent cubic forms vanishing at all 8 points. For convenience, we choose $u$ to be the form $(x+z)^3-(y+z)^2z$. We let $w$ be a sextic form with double roots at each \newtext{of the points listed above; we further require that $w$ is not in the span of $\set{u^2,uv,v^2}$. Note that $\riemannrochspace{-2\kappa_X}$ is spanned by $u^2,uv,v^2,w$. In our computations in \cite{kulkarni2016magmascript} we also make explicit choices for $v$ and $w$, but it is not particularly helpful to highlight them.} Thus the zero-locus of 
		\[
			\newtext{F} := \det
				\bbm
				u_x & v_x & w_x \\
				u_y & v_y & w_y \\
				u_z & v_z & w_z
				\ebm 
		\]
	is a (singular) model of the branch curve of $X$. With $\varphi(x,y,z) := (u:v:w)$ as before, we let $t := u/v$ and $W := w$. Using the MAGMA \cite{MAGMA} computer algebra package we can compute the image of $C$ under $\varphi$ and write the defining equation in terms of $t,W$. We have that
	\begin{alignat*}{4}
		C\:0=&	  	&&23200074887895098984232713028 \ t^6 - 2457892462046662336694429 \ t^5 + \\
		&	  	&&1338378986926042827721/16 \ t^4 - 9000960055643209/8 \ t^3 + 			\\
		&	  	&&158059424789/16 \ t^2 + 11025 \ t										\\
		&+W     &&(24403582287284966245 \ t^4 - 13786310912398097/8 \ t^3 +	 			\\
		&	  	&& \ \ 234505995159/8 \ t^2 - 316801/4 \ t)									\\
		&+W^2   &&(136902207241/16 \ t^2 - 1208223/4 \ t)								\\
		&+W^3 	&&
	\end{alignat*}
	is an affine model for the branch curve of $X$. \newtext{The morphism induced by the projection $(t,W) \mapsto t$ is the degree $3$ morphism $f\: C \ra \mbp_{k}^1$ from Proposition \ref{del-pezzo}(b). Our choice of $u$, via Corollary \ref{cor: fibres and ramification of branch curve}, ensures that} there is a totally ramified fibre at $t=0$. \newtext{We may view the fibres of $f$ as effective degree $3$ divisors of $C$. By the discussion above and Proposition \ref{del-pezzo}(d), we see that the divisor class of the fibre at $t=0$ (and hence any fibre of $f$) is an even theta characteristic.} We use the representatives of the theta characteristics to compute data presented in Example \ref{ex:main example 2}.
\end{example}

\section{Recovering large class groups from curves with large rational 2-torsion}

We fix some notation to be used for the remainder of this section. Using the procedure in the previous section we may choose a trigonal morphism $f \: C \ra \mbp_\mbq^1$ with $C$ a genus 4 curve such that $\ranktwo J_C[2](\mbq) = 8$ and let $P_0$ be a rational totally ramified point of $f$. Denote $p_0 := f(P_0)$. Let $S$ be the places of bad reduction for $C/\mbq$ as well as the archimedean places and places dividing $2$.

Conveniently, the Abel-Jacobi map with base-point $P_0$ provides an embedding of $C$ into its Jacobian. This map, which we denote by $j$, is defined over $\mbq$ and is given by $j(Q) := [Q-P_0]$. In particular, $j(P_0)$ is the identity point of $J_C$. 

\newtext{Corollary 2.11 of \cite{gile12} already allows us to conclude for these curves that for all but $O(\sqrt{N})$ of the $t$ in $\{1, \ldots, N \}$ that $[\mbq(P_t) : \mbq] = 3$ and
	\[
		\ranktwo \Cl(\mbq(P_t)) \geq 8 + \#S - \rank \mbo_{\mbq(P_t),S}^*
	\]
where $P_t := f^{-1}(t)$,} $S$ is the set of bad primes for $C$ together with the archimedean primes and those dividing $2$, \newtext{and $\mbo_{\mbq(P_t),S}^*$ is the group of $S$-units in the number field $\mbq(P_t)$}. By applying the ideas of \cite{mes83} to the curves arising from our construction it is possible to avoid the penalty on the bound introduced by $S$-units. The overarching idea of the method we use is to directly exhibit a subfield of the Hilbert class field of $\mbq(P)$, where $[\mbq(P):\mbq]=3$ and $f(P) \in \mbp^1(\mbq)$, for $P \in C(\mbq^\al)$ which lie in fibres that satisfy a local condition at finitely many places, thereby demonstrating that the class groups of the fields associated to these fibres have large 2-rank. Bilu and Gillibert have provided a generalized description of this framework in \cite{bilu2016chevalley}. Nevertheless, we provide proofs as our computations closely mirror the arguments.

The main idea introduced by \cite{mes83} is to use the fact that the fibres of the multiplication-by-$2$ morphism of the Jacobian of a curve $C$ are precisely described. Specifically:

\begin{lemma} \label{lem:extension is Galois}
Let $K$ be a number field, let $\Spec K \in J_C(K)$ be a closed point of degree $[K:\mbq]$ on $J_C$ and let $\Spec L$ be its inverse image under $[2]\: J_C \ra J_C$. If $J_C[2](\mbq) \iso (\mbz/2)^{2g}$, then $L/K$ is Galois with Galois group $(\mbz/2)^{2g}$ and \newtext{there is a fixed finite set of places $S$ of $K$, independent of $L$, such that} $L/K$ unramified outside of $S$.
\end{lemma}

In fact, we can say even more; the ramification at the bad places is explicitly described by the descent method (as we explain below). 

\begin{proposition}[\protect{\cite[Proposition C.1.5]{sh00}}] \label{good places}
Let $k$ be a number field and $\wtilde S$ be the finite set of places of $k$ of bad reduction for $C$, as well as primes dividing $2$. Then for any $y \in J_C(k^{\al})$ such that $[2]y \in J_C(k)$ we have that $k(y)/k$ is unramified outside of $\wtilde S$.
\end{proposition}

We now explain the computations involved at the bad places. Let $A := J_C$ be the Jacobian of a curve of genus $g$ and recall that there is a pairing 
\[
\gen{\cdot,\cdot} \: \frac{A(k)}{2} \times A[2](k) \ra k^*/k^{*2}
\]
induced from the Weil pairing \cite[Section 2.2]{schaefer1998computing}. This pairing is explicitly computable given representatives for a basis $T_1,\ldots,T_{2g}$ of $A[2]$ as follows. As each $T_i$ is 2-torsion we have that there is a rational function $h_i$ such that $\div h_i = 2T_i$. For $[D] = \left[ \sum_{P \in C(k^{\al})} n_P P \right] \in A(k)/2$, where $n_P$ is zero for points of $C$ which occur in the support of $T_i$ (i.e the zeros and poles of $h_i$) we can define
\[
\left\langle \left[ \sum_{P \in C(k^{\al})} n_P P \right], [T_i] \right\rangle := \prod_{P \in C(k^{\al})} h_i(P)^{n_P}.
\]
Note that while the value of each $h_i(P)$ lies in $k^{\al}$ the product lies in $k^*$ for any $[D] \in \Pic^0(C)(k)$. We also note that every divisor class $[D]$ has a representative of the form $\sum_{P \in C(k^{\al})} n_P P$ where $n_P$ is zero for points of $C$ which occur in the support of $T_i$ \cite[Lemma A.2.3.1]{sh00}. The value of $\prod_{P \in C(k^{\al})} h_i(P)^{n_P}$ depends on the chosen representatives for the $2$-torsion and the representative of $[D]$ in $A(k)/2$. However, the square class of this value is independent of these choices so the pairing is well defined \cite[Lemma 2.1]{schaefer1998computing}. 

 We are ready to state the well-known companion result to Proposition \ref{good places} for the bad places. We provide a proof lacking an immediate reference.
\begin{proposition} \label{bad places}
	Let $C$ be a curve of genus $g$ defined over $\mbq$ with $\rank_2 J_C(\mbq)[2] = 2g$. Let $k$ be a number field, let $A := J_C$, let $\set{T_1, \ldots, T_{2g}}$ be a basis for $A[2](\mbq)$, and let $h_1, \ldots, h_{2g}$ be the maps $h_i \: A(k)/2 \ra k^*/k^{*2}$ defined from the Kummer pairing. Let $x \in A(k)$ and choose $y \in A(k^{\al})$ such that $[2]y = x$. Then $k(y)/k$ is unramified if both of the following hold:
	\begin{enumerate}[(a)]
		\item $\ord_\nu h_i(x) \equiv 0 \pmod 2$ for each $i$ and each finite place $\nu$ of $k$.
		\item Each $h_i(x)$ has a positive representative modulo $k^{*2}$ at all real places.
	\end{enumerate} 
\end{proposition}

\begin{proof}
	\newtext{Let $S$ be the set of places of $\mbq$ of bad reduction for $C$ together with the archimedean places and places dividing 2, and let $\wtilde S$ be the set of places of $k$ over places in $S$.} Proposition \ref{good places} shows that it suffices to consider only those places contained in $\wtilde S$. Let $\nu \in \wtilde S$ be a finite place and let $\sigma \in I_\nu$ be an element in the inertia subgroup of $\Gal(k^{\al}/k)$. Then we need to show that $[\sigma y - y] = [0]$. Observe that by definition of $y$ we have $[\sigma y - y] \in A[2]$. Now
		\[
			\gen{[\sigma y - y], T_i} = \frac{h_i(y)^\sigma}{h_i(y)}
		\]
	Note again since $[2]y = x$ that $h_i(y)^2 = h_i(x) \pmod {k^{*2}}$. Thus $\ord_\nu h_i(y) \in \mbz$ so the inertia group acts trivially on $h_i(y)$. In particular each $\gen{[\sigma y - y], T_i}$ is trivial.
	
	For archimedean places $\nu$ of $\wtilde S$, we have that the inertia group $I_\nu$ is trivial if $\nu$ a complex place. If $\nu$ is a real place then we have assumed each $h_i(x)$ has a positive representative at $\nu$, so 
		\[
		\gen{[\sigma y - y], T_i} = \frac{h_i(y)^\sigma}{h_i(y)}
		\]
	is always trivial for $\sigma \in I_\nu$. Since the Kummer pairing has trivial left kernel we see that $I_\nu$ acts trivially.
\end{proof}

In what follows we devote ourselves to ensuring the conditions of Proposition \ref{bad places} are met. We will need to apply Proposition \ref{bad places} to points on $C$ defined over different cubic number fields. As these number fields vary, so too does the set of places $\wtilde S$ extending the places of $\mbq$ of bad reduction. We will use the totally ramified point of $f\:C \ra \mbp^1$ to accommodate for this variation.

Fix a basis $T_1,\ldots,T_8$ for $J_C[2]$ and rational functions $h_1,\ldots,h_8$ corresponding to the Kummer pairing. We note that once we find divisors representing $T_1,\ldots,T_8$, we can write down each $h_i$ explicitly, and up to changing the divisors representing the $T_i$ we can assume that each $h_i$ is regular at $P_0$.

\newtext{At this point it is helpful to let $t \in k(C)$ be the function field element corresponding to $f$. Note that if $C$ is the curve from Example \ref{ex:main example} then this assignment agrees with the assignment of $t$ from that example.} Let $\alpha$ be a uniformizing element for $\mbo_{C,P_0}$. Then the image of $P \mapsto (t(P), \alpha(P))$ is an affine plane model of $C$ which is non-singular at $P_0$. For convenience we will denote
	\[
		\norm{(x_1,y_1) - (x_2,y_2)}_w := \max \set{\abs{x_1-x_2}_w, \abs{y_1-y_2}_w}
	\]
for a place $w$ of a number field $K$ and points $(x_1,y_1), (x_2,y_2) \in \mba^2(K)$. 

\begin{lemma} \label{lem:ram near p}
	Let $\nu \in S$ be a finite place. Then there is a constant $\ell_\nu$ depending only on $f$, $\nu$, $P_0$, a uniformizer $\alpha \in \mbo_{C,P_0}$, and the associated affine model of $C$ such that the following statement holds:
		
		If $q \in \mbp^1(\mbq)$ satisfies $\abs{q-p_0}_\nu < \ell_\nu$, then $\abs{h_i([Q-P_0]) - 1}_w < 1$ for every $Q \in f^{-1}(q)(\mbq^{\al})$, for every $i \in \set{1, \ldots, 8}$, and for each $w$ a place of $\mbq(Q)$ extending $\nu$. In particular, if $w$ is a real place, then $h_i([Q-P_0])$ is positive.
\end{lemma}

\begin{proof}
	As each $h_i$ is a rational function on $C$ with coefficients in $\mbq$ and regular at $P_0$, we see that there is a constant $\lambda$ depending only on the affine plane model of $C$ such that for any place $w$ of a number field $K/\mbq$ we have that 
	\[
	Q \in C(K) \cap \mba^2(K) \text{ and } \norm{Q-P_0}_w < \lambda \implies \abs{h_i([Q-P]) - 1}_w < 1
	\]
	for each $1 \leq i \leq 8$. One way to see this is to write $h_i \in \mbo_{C,P_0}$ as a power series in $\alpha$. 
	
	Let $\nu \in S$ and let
		\[
			C_{\aff} \: 0 = W^3 + a(t) W^2 + b(t) W + c(t)
		\]
	be the affine plane model which is the image of $P \mapsto (t(P), \alpha(P))$, where we may assume that $P_0$ maps to the origin and $t(P_0) = 0$. Since the fibre over $t(P_0)$ contains a unique point and $P_0$ is a non-singular point of the model we may rewrite our model as 
		\[
			C_{\aff} \: 0 = W^3 + t a(t) W^2 + tb(t) W + tc(t)
		\]
	with $a(t),b(t),c(t)$ regular at \newtext{$P_0$}. By taking $\abs{t_1}_\nu$ less than $\lambda$ and sufficiently small in terms of the coefficients of $a(t),b(t), c(t)$ we see that the slopes of the Newton polygon of
		\[
			0 = W^3 + t_1 a(t_1) W^2 + t_1 b(t_1) W + t_1 c(t_1)
		\]
	will all be at least $-\frac{1}{3} \log_p \lambda$. In particular, all of the roots have $\norm{\beta_i}_w < \lambda$ for every place $w$ of $\mbq(Q)$ lying above $\nu$. So each $\norm{(t_1,\beta_i) - (0,0)}_w < \lambda$ and we are done.
\end{proof}

\begin{example} \label{ex:main example 2}
	We provide a MAGMA script, available at \cite{kulkarni2016magmascript}, which shows how to explicitly compute an appropriate $\ell_\nu$ to apply Lemma \ref{lem:ram near p} to Example \ref{ex:main example}. Note that the defining equation for $C_{\aff}$ has integral coefficients up to powers of $2$. 
	
	In characteristic greater than $5$, the branch curve of any degree $1$ del Pezzo surface is smooth (cf. the proof of \cite[Proposition 3.2]{vak01}). Thus, any prime $p$ for which the reductions of the $8$ base points (considered as points of $\mbp^2_{\mbf_p}$) remain in general position is a prime of good reduction for $C$. Using MAGMA we can compute that the primes of bad reduction for $C$ are contained in
		\allstar{
			S' := \{ &2, 3, 5, 7, 11, 13, 17, 19, 23, 29, 31, 37, 41, 43, 47, 59, 61, 71, \\ 
			&83, 103, 107, 179, 223, 241, 389, 449, 599, 809, 1019 \}.
		}
	At each prime $p$ which is not $2$ or $3$, we can use MAGMA to compute a regular model for the curve at $p$. More specifically, we can compute the component group at the prime $p$. By \cite[Proposition 3.2]{schaefer2004howto} we may remove from consideration all primes in $S'$ where the component group has odd order (Note that the statement in \cite{schaefer2004howto} is for elliptic curves, but the result is valid for Jacobians of curves). In this particular example we do not remove any further primes.
		
	\newtext{Recall from Example \ref{ex:main example} that the effective degree $3$ divisor of $C_\aff$ given by the places of $k(C_\aff)$ over $t=\infty$ is a representative of an even theta characteristic of $C$; we call this divisor $O_E$.} The function associated to the divisor $[\Theta_1 - O_E]$, with $\Theta_1$ the divisor corresponding to the places of $k(C_{\aff})$ over $(3:7:1)$, is given on the affine model (up to a square constant) by
		\allstar{
			h(t,W) & := && -484335370397555869540982096 \ t^2 + 21745428828566997697489 \ t - \\
				   &	&&	184765518741585604 \ W + 22709411000816400
		}
	 \newtext{(recall that $C_\aff$ is defined via the blow-up of a singular plane curve)}. One can apply the explicit method of Lemma \ref{lem:ram near p} to find the required sufficiently small constants. These are listed below.
	\begin{center}
		\begin{tabular}{c|c}
			\text{Place} & $-\log_p \ell_p$ \\ \hline
			$2$ & $33$ \\
			$3$ & $21$ \\ 
			$5$ & $21$ \\ 
			$7$ & $21$ \\
			$11$  & $5$ \\ 
			$13$  & $5$ \\
			$17$  & $1$ \\
			$19$  & $5$ \\
			$23$  & $9$ \\
			$29$  & $1$ \\
			$31$  & $5$ \\
			$37$  & $1$ \\
			$41$  & $1$ \\
			$43$  & $5$ \\			
			$47$  & $1$ \\
		\end{tabular}
		\quad
		\begin{tabular}{c|c}
			\text{Place} & $-\log_p \ell_p$ \\ \hline
			$59$  & $1$ \\
			$61$  & $1$ \\
			$71$  & $1$ \\
			$83$  & $1$ \\
			$103$ & $1$ \\
			$107$ & $1$ \\
			$179$ & $5$ \\
			$223$ & $1$ \\
			$241$ & $1$ \\
			$389$ & $1$\\
			$449$ & $5$ \\
			$599$ & $5$ \\
			$809$ & $5$ \\
			$1019$ & $1$ \\
			\ & \			 
		\end{tabular}
	\end{center}
\end{example}

\begin{lemma} \label{lem:real ram near p}
	Let $\nu \in S$ be an archimedean place, let $q \in \mbp^1(\mbq)$, and let $Q \in f^{-1}(q)(\mbq^{\al})$. Then there is a constant $\ell_\nu$ depending only on $f$, $\nu$, $P_0$, a uniformizer $\alpha \in \mbo_{C,P_0}$, and the associated affine model of $C$ such that
	\[
	\abs{q-p_0}_\nu < \ell_\nu \implies \abs{h_i([Q-P_0]) - 1}_w < 1
	\]
	for each $w$ a place of $\mbq(Q)$ extending $\nu$. 
\end{lemma}

\begin{proof}
	The place $w$ determines an embedding $C_{\aff}(\mbq(Q)) \inj C_{\aff}(\mbc)$. On the complex points we have that the functions
		\[
			\abs{h_i} \: C_{\aff}(\mbc) \ra \mbr
		\]
	are continuous and positive in a small neighbourhood $U$ around $P_0$. Since $P_0$ is a totally ramified point of $f$ it follows that the pullback of a small interval $B$ containing $p$ will be contained in $U$. Since the embedding of $P_0$ into $C_{\aff}(\mbc)$ does not depend on $w$ we see that $U$ and $B$ are independent of $w$ as well. 
\end{proof}

\newtext{
We will use the result of Bilu and Gillibert, based on a result of Dvornicich and Zannier \cite{dvornicich1994fields} and Hilbert's Irreducibility Theorem, to ensure that there are infinitely many non-isomorphic cubic number fields in the family we have constructed. Here, we have specialized the statement of \cite[Theorem 3.1]{bilu2016chevalley} to the particular case where the number field in question is $\mbq$. We follow the terminology of \cite{bilu2016chevalley}.

\begin{definition}
	We call $\mho \ssq \mbq$ a basic thin subset of $\mbq$ if there exists a smooth geometrically irreducible curve $C$ defined over $\mbq$ and a non-constant rational function $u \in K(C)$ of degree at least $2$ such that $\mho \ssq u(C(\mbq))$. A thin subset of $\mbq$ is a union of finitely many basic thin subsets.
\end{definition}

\begin{remark} \label{rem: hilbert irreducibility non-quantitative}
	Let $f\: C \ra \mbp_{\mbq}^1$ be a morphism of degree $d$ greater than $1$. The set of $\alpha \in \mbq$ such that $\mbq(P_t)$ is not a number field of degree $d$ over $\mbq$, where $P_t = f^{-1}(\alpha : 1)$, is a thin subset of $\mbq$ \cite[Proposition 3.5]{bilu2016chevalley}. 
\end{remark}

\begin{theorem}[ {\cite[Theorem 3.1]{bilu2016chevalley}} ] \label{thm: dvornicich-zannier-bilu-gillibert}
	Let $C$ be a curve over $\mbq$ and let $t \: C \ra \mbp^1$ be a non-constant morphism with $\deg t > 1$. Let $S$ be a finite set of places of $\mbq$ possibly containing the place at infinity. Further, let $0 < \epsilon \leq 1/2$ and let $\mho$ be a thin subset of $\mbq$. Then there exist positive numbers $c = c(\mbq, C, t, S, \epsilon)$ and $B_0 = B_0(\mbq, C, t, S, \epsilon, \mho)$ such that for every $B \geq B_0$ the following holds. Consider the points $P \in C(\mbq^\al)$ satisfying 
	\begin{align*}
		t(P) &\in \mbq \bs \mho, \\
		\abs{t(P)}_\nu &< \epsilon \hspace{0.3in} (\nu \in S), \\
		H(t(P)) &\leq B.
	\end{align*}
	Then among the number fields $\mbq(P)$, where $P$ satisfies the conditions above, there are at least $cB/\log B$ distinct fields of degree $d$ over $\mbq$.
\end{theorem}

\begin{corollary} \label{cor: main result}
	Let $C$ be the curve from Example \ref{ex:main example 2}, let $f\: C \ra \mbp_{k}^1$ be the trigonal morphism from Example \ref{ex:main example}, and $S$ the set of places from Example \ref{ex:main example 2} together with the archimedean place. By Lemma \ref{lem:ram near p} and Lemma \ref{lem:real ram near p} choose constants $\ell_\nu$ for each $\nu \in S$. Let $\epsilon = \min\{ \{\ell_\nu : \nu \in S \}, 1/2 \}$. 
	
	\bigskip
	
	Let $T := \set{t_1,\ldots,t_r}$  enumerate the points of $\mba^1(\mbq)$ of height less than $B$ which satisfy $\abs{t_i}_{\nu} < \epsilon$ for each $\nu \in S$, and let $P_{t_i} := f^{-1}(t_i)$ be the corresponding fibres. Then aside from a thin set of exceptions, we have that $\mbq(P_{t_i})$ is a cubic extension of $\mbq$ with $\ranktwo \Cl(\mbq(P_{t_i}))[2] \geq 8$. Moreover, letting $\eta(B)$ be the number of isomorphism classes of number fields in the set $\set{\mbq(P_t) : t \in T}$, we have that $\eta(B) \gg \frac{B}{\log B}$. 
\end{corollary}

}

\begin{proof}
	\newtext{Let $X := [2]^{-1}C$ be the (irreducible) unramified degree $2^8$ cover of $C$ obtained from pulling back $C$ along $[2]$ in $J_C$. By applying Remark \ref{rem: hilbert irreducibility non-quantitative}} to the composition $g \: X \overset{[2]}{\longrightarrow}  C \overset{f}{\longrightarrow} \mbp^1$ we see, aside from \newtext{a thin set of exceptions $\mho$}, that $\newtext{\mbq(g^{-1}(t))}$ is a degree $2^8 \cdot 3$ extension of $\mbq$. In particular $\newtext{\mbq(P_t) := \mbq(f^{-1}(t))}$ is a cubic extension of $\mbq$ and $\mbq([2]^{-1}(P_t))/\mbq(P_t)$ is a degree $2^8$ field extension. Additionally, the number of isomorphism classes of cubic number fields in $\set{\mbq(P_t) : t \in T}$ is an immediate consequence of \newtext{Theorem \ref{thm: dvornicich-zannier-bilu-gillibert}}.
	
	For the remainder of the argument fix such a $t \in T$ and denote $[2]^{-1} P_t := \Spec K$. By Proposition \ref{good places} we see that the extension $K/\mbq(P_t)$ is unramified outside of places lying above those in $S$. Moreover, $K/\mbq(P_t)$ is Galois with Galois group $(\mbz/2)^8$ by Lemma \ref{lem:extension is Galois}. 
	
	By Lemma \ref{lem:ram near p} we have chosen $\ell_\nu$ for each finite $\nu \in S$ such that whenever $|t-p_0|_\nu < \ell_\nu$ we have 
	\[
	|h_i([P_t-P_0]) - 1|_w < 1
	\] 
	for every place of $\mbq(P_t)$ over $\nu$. But $\ord_w h_i([P_t - P_0]) = 0$, so by Proposition \ref{bad places}	we conclude that $K/\mbq(P_t)$ is unramified at every finite place of $\mbq(P_t)$. To ensure that these extensions are unramified at the archimedean places we apply Lemma \ref{lem:real ram near p}. By class field theory we conclude that $\ranktwo \Cl (\mbq(P_t))[2] \geq 8$.
\end{proof}

\begin{remark}
We see that we obtain an improvement of the bound presented in \cite{gile12}. The quantity $\rank \mbo_{\mbq(P),S}^* - \# S$ which occurs in their estimate is minimized at $\left[ \frac{\deg f - 1}{2} \right] = 1$, giving a lower bound of $7$ for the $2$-rank of the class groups in a family.
\end{remark}

\section{Families of Galois cubic fields from genus 4 curves}

In Section 3 we obtained families of cubic number fields whose class groups have high 2-rank from genus 4 curves which were not geometrically special aside from the fact that they are uniquely trigonal \cite{vak01}. It is natural to ask if we could modify our construction to create families of special cubic number fields that have class groups with large 2-rank. 

For instance, it might be possible to produce many rational $2$-torsion classes on a curve with an affine plane model of the form $w^3=f(t)$, thereby giving rise to a family of cubic number fields that become Galois after joining a primitive third root of unity and have a class group with large 2-rank. In this section, we prove Proposition \ref{prop:no rank 8}, which shows that it is impossible to find genus 4 curves which are a $\mu_3$-branched cover of $\mbp^1$ with fully rational $2$-torsion.

\begin{definition}
	We say that a curve $C$ is a \supellname if there exists a (possibly singular) affine plane model for $C$ of the form $w^3 = f(t)$ with $f(t)$ non-constant. \newtext{In general, we say that a morphism of curves $\pi\: C \ra B$ defined over $\mbq$ is a $\mu_3$-branched cover if there is a non-constant $f(t) \in \mbq(B)$ such that $\mbq(C) = (\pi^* \mbq(B))(\sqrt[3]{f})$.}
\end{definition}

\begin{proposition} \label{prop:no rank 8}
	Let $C/\mbq$ be a \supellname of genus $4$. Then \\ $\ranktwo J_C[2](\mbq) \leq 6$.
\end{proposition}

Notice that for curves satisfying the hypotheses of Proposition \ref{prop:no rank 8} that Galois acts both on the $[2]$-torsion of $C$ and on $\Aut(C/\mbp_\mbq^1) \iso \mu_3$. In order to prove the proposition we need to assemble all of these group actions together in a nice way. 

	\begin{lemma}
	Let $\pi \: C \ra B$ be a $\mu_3$-branched covering of curves defined over $\mbq$. Let $\rho\: G_\mbq \ra \Sp(2g(C),2)$ be the representation of Galois arising from the action of $G_\mbq$ on $J_C[2]$ and define the action of $\tau \in G_\mbq$ on $g \in \Sp(2g(C),2)$ by
	\[
	\act{\tau}{g} := \rho(\tau) g \rho(\tau)^{-1}.
	\]
	Then the representation $\psi\: \Aut(C/B) \ra \Sp(2g(C),2)$ corresponding to the action of $\Aut(C/B)$ on divisor classes commutes with the action by Galois. i.e
	\[
	\psi( \act{\tau}{\sigma}) = \rho(\tau) \psi(\sigma) \rho(\tau)^{-1}
	\]
	for all $\tau \in G_\mbq$ and $\sigma \in \Aut(C/B)$.
	\end{lemma}
	
	\begin{proof}
		Since $\pi$ is a $\mu_3$-branched cover defined over $\mbq$, there is an $f \in \mbq(B)$ such that $\mbq(C) = \mbq(B)(\sqrt[3]{f})$ and \newtext{such that for some $\sigma$ generating $\Aut(C/B)$ we have that} $\sigma^*$ acts on $\sqrt[3]{f}$ by multiplication by $\zeta_3$. Let $\Gamma \ssq \mbp^1 \times B$ be the possibly singular model of $C$ given by the image of 
			\[
				\phi\: P \mapsto (\sqrt[3]{f}(P),\pi(P)).     
			\]
			Conveniently, $\sigma$ is explicitly described on points of $\Gamma$ as $\sigma\: (w,t) \mapsto (\zeta_3 w,t)$. We see that $(\act{\tau}{\sigma})P = (\tau\sigma\tau^{-1})P$ for every $\mbq^{\al}$-point $P$ of $\Gamma$, so $\act{\tau}{\sigma}$ and $\tau\sigma\tau^{-1}$ have the same action on the divisors of $\Gamma$. 
			Let $\Delta$ be the singular locus of $\Gamma$. By \cite[Lemma A.2.3.1]{sh00} every divisor on $C$ is linearly equivalent to a divisor $D' = \sum_{P \in C(\mbq^\al)}n_P P$ such that $n_P=0$ for each $P \in \phi^{-1}(\Delta)(\mbq^\al)$.
		
		Thus the actions determined by $\act{\tau}{\sigma}$ and $\tau\sigma\tau^{-1}$ are equal on divisor classes of $C$. (This is true despite the fact that in general, $(\Pic C)^{G_\mbq} \neq \Div(C)(\mbq)/\Princ(C)$).
	\end{proof}

	\begin{lemma} \label{lem: torsion in mu3 cover}
		Let $\pi\: C \ra B$ be a $\mu_3$-branched cover of curves defined over $\mbq$ with $g(C) > 1$. Then
			\begin{enumerate}[(i)]
				\item $\ranktwo J_C[2](\mbq) < 2g(C)$.
				\item If in addition $B$ has genus 0, then $\ranktwo J_C[2](\mbq) \leq \log_2 \frac{2^{2g(C)}+2}{3}$. 
			\end{enumerate}
	\end{lemma}

	\begin{proof}
		\begin{enumerate}[(i)]
			\item
			Let $\sigma$ be a non-trivial element of $\Aut(C/B)$ and let $D \in J_C[2](\mbq)$. Note that the image of $\pi^*\: J_B[2] \ra J_C[2]$ consists of exactly the $2$-torsion classes fixed by $\sigma$. Indeed if $D$ is $\sigma$-stable then
				\[
					D = [3]D = D + \sigma D + \sigma^2 D \in \pi^*(J_B).
				\]
			Since the genus of $C$ is greater than $1$, we have that $g(B) < g(C)$ and hence that $\pi^*$ is not surjective. Thus we either have $\ranktwo J_C[2](\mbq) < 2g(C)$ or we can find a $D \in J_C[2](\mbq)$ which is not $\sigma$-invariant. But now
				\[
				\rho(\tau)\sigma(D) = \rho(\tau) \sigma \rho(\tau)^{-1}(D) = {\act{\tau}{\sigma}(D)} = \sigma^{-1}(D) \neq \sigma(D)
				\]
			so $J_C[2]$ is not fully rational.	
			
			\item
			Let $D$ be a non-trivial $2$-torsion class of $J_{C}[2]$. \newtext{We see from the calculation in part (i) that if $B = \mbp_\mbq^1$ then $D$ cannot be fixed by $\sigma$.} Since $D$ is not $\sigma$-stable it follows by the argument in part $(i)$ that at most one of $D,\sigma D, \sigma^2 D$ is rational. The quantity $\frac{2^{2g(C)}+2}{3}$ is the number of $\gen{\sigma}$-orbits in $J_C[2]$.
		\end{enumerate} 
	\end{proof}

	Proposition \ref{prop:no rank 8} follows immediately from Lemma \ref{lem: torsion in mu3 cover}. We note that it is possible to construct a genus 4 curve $C$ which is a $\mu_3$-branched cover of another curve $B$ and has $\ranktwo J_C[2](\mbq) = 6$ using del Pezzo surfaces. The eight points
		\begin{gather*}
			 (0:1:0),(0:0:1), (1:1:1), (1:\zeta_3:\zeta_3^2),  \\
				(1:\zeta_3^2:\zeta_3), (3:4:5), (3:4 \zeta_3:5 \zeta_3^2), (3:4 \zeta_3^2:5 \zeta_3)
		\end{gather*} 
	are in general position, invariant under a linear $\mu_3$ automorphism of $\mbp^2$, and have six distinct Galois orbits. Consequently, the associated genus 4 curve $C$ has $\mu_3 \ssq \Aut_{\mbq^\al}(C)$ and $\ranktwo J_C[2](\mbq) = 6$. However, $C/\mu_3$ is not isomorphic to $\mbp^1$. We have yet to find an example of a \supellname of genus 4 whose 2-rank is 6.
	
	\begin{question}
		Is there a \supellname of genus four $C$ defined over $\mbq$ with $\rank_2 J_C[2](\mbq)$ equal to 6?
	\end{question}


\begin{thebibliography}{BCP97}
	
	\bibitem[AC55]{ankeny1955divisibility}
	N.~C. Ankeny and S.~Chowla.
	\newblock On the divisibility of the class number of quadratic fields.
	\newblock {\em Pacific J. Math.}, 5:321--324, 1955.
	
	\bibitem[BCP97]{MAGMA}
	Wieb Bosma, John Cannon, and Catherine Playoust.
	\newblock The {M}agma algebra system. {I}. {T}he user language.
	\newblock {\em J. Symbolic Comput.}, 24(3-4):235--265, 1997.
	\newblock Computational algebra and number theory (London, 1993).
	
	\bibitem[BG16]{bilu2016chevalley}
	Yuri Bilu and Jean Gillibert.
	\newblock Chevalley-weil theorem and subgroups of class groups.
	\newblock {\em arXiv preprint arXiv:1606.03128}, 2016.
	
	\bibitem[BL05]{bilu2005divisibility}
	Yuri~F. Bilu and Florian Luca.
	\newblock Divisibility of class numbers: enumerative approach.
	\newblock {\em J. Reine Angew. Math.}, 578:79--91, 2005.
	
	\bibitem[CL84]{cohen-lenstra}
	H.~Cohen and H.~W. Lenstra, Jr.
	\newblock Heuristics on class groups of number fields.
	\newblock In {\em Number theory, {N}oordwijkerhout 1983 ({N}oordwijkerhout,
		1983)}, volume 1068 of {\em Lecture Notes in Math.}, pages 33--62. Springer,
	Berlin, 1984.
	
	\bibitem[DZ94]{dvornicich1994fields}
	R.~Dvornicich and U.~Zannier.
	\newblock Fields containing values of algebraic functions.
	\newblock {\em Ann. Scuola Norm. Sup. Pisa Cl. Sci. (4)}, 21(3):421--443, 1994.
	
	\bibitem[GL12]{gile12}
	Jean Gillibert and Aaron Levin.
	\newblock Pulling back torsion line bundles to ideal classes.
	\newblock {\em Math. Res. Lett.}, 19(6):1171--1184, 2012.
	
	\bibitem[Hon68]{honda1968real}
	Taira Honda.
	\newblock On real quadratic fields whose class numbers are multiples of {$3$}.
	\newblock {\em J. Reine Angew. Math.}, 233:101--102, 1968.
	
	\bibitem[HS00]{sh00}
	Marc Hindry and Joseph~H. Silverman.
	\newblock {\em Diophantine geometry}, volume 201 of {\em Graduate Texts in
		Mathematics}.
	\newblock Springer-Verlag, New York, 2000.
	\newblock An introduction.
	
	\bibitem[Kul16]{kulkarni2016magmascript}
	Avinash Kulkarni.
	\newblock Magma script accompanying this paper, 2016.
	\newblock Available at \url{http://www.sfu.ca/~akulkarn/delPezzoBranchCurve.m}.
	
	\bibitem[Mes83]{mes83}
	Jean-Francois Mestre.
	\newblock Courbes elliptiques et groupes de classes d'id\'eaux de certains
	corps quadratiques.
	\newblock {\em J. Reine Angew. Math.}, 343:23--35, 1983.
	
	\bibitem[Mur99]{murty1999exponents}
	M.~Ram Murty.
	\newblock Exponents of class groups of quadratic fields.
	\newblock In {\em Topics in number theory ({U}niversity {P}ark, {PA}, 1997)},
	volume 467 of {\em Math. Appl.}, pages 229--239. Kluwer Acad. Publ.,
	Dordrecht, 1999.
	
	\bibitem[Nag22]{nagel1922klassenzahl}
	Trygve Nagell.
	\newblock \"{U}ber die {K}lassenzahl imagin\"ar-quadratischer {Z}ahlk\"orper.
	\newblock {\em Abh. Math. Sem. Univ. Hamburg}, 1(1):140--150, 1922.
	
	\bibitem[Nak86]{nakano1986construction}
	Shin Nakano.
	\newblock On the construction of pure number fields of odd degrees with large
	{$2$}-class groups.
	\newblock {\em Proc. Japan Acad. Ser. A Math. Sci.}, 62(2):61--64, 1986.
	
	\bibitem[Sch98]{schaefer1998computing}
	Edward~F. Schaefer.
	\newblock Computing a {S}elmer group of a {J}acobian using functions on the
	curve.
	\newblock {\em Math. Ann.}, 310(3):447--471, 1998.
	
	\bibitem[SS04]{schaefer2004howto}
	Edward~F. Schaefer and Michael Stoll.
	\newblock How to do a {$p$}-descent on an elliptic curve.
	\newblock {\em Trans. Amer. Math. Soc.}, 356(3):1209--1231, 2004.
	
	\bibitem[Vak01]{vak01}
	Ravi Vakil.
	\newblock Twelve points on the projective line, branched covers, and rational
	elliptic fibrations.
	\newblock {\em Math. Ann.}, 320(1):33--54, 2001.
	
	\bibitem[Wei73]{weinberger1973quadratic}
	P.~J. Weinberger.
	\newblock Real quadratic fields with class numbers divisible by {$n$}.
	\newblock {\em J. Number Theory}, 5:237--241, 1973.
	
	\bibitem[Yam70]{yamamoto1970unramified}
	Yoshihiko Yamamoto.
	\newblock On unramified {G}alois extensions of quadratic number fields.
	\newblock {\em Osaka J. Math.}, 7:57--76, 1970.
	
	\bibitem[Zar08]{zar08}
	Yu.~G. Zarhin.
	\newblock Del {P}ezzo surfaces of degree 1 and {J}acobians.
	\newblock {\em Math. Ann.}, 340(2):407--435, 2008.
	
\end{thebibliography}

\end{document}